\documentclass[a4paper,10pt]{article}
\usepackage{a4}
\usepackage[ansinew]{inputenc}
\usepackage{epsfig}
\usepackage{amssymb}
\usepackage{amsmath}
\usepackage{amsthm} 
\usepackage{amscd}
\usepackage{enumerate}
\usepackage{float}
\usepackage{xpatch}

\makeatletter
\xpatchcmd{\proof}{\@addpunct{.}}{\normalfont\,\@addpunct{:}}{}{}
\makeatother

  \newtheoremstyle{dotless}{}{}{\itshape}{}{\bfseries}{:}{ }{}

    \newtheoremstyle{dotlessrem}{}{}{}{}{\bfseries}{:}{ }{}

  \theoremstyle{dotless}

\newtheorem{lem}{Lemma}[section]
\newtheorem{theo}[lem]{Theorem}
\newtheorem*{cor}{Corollary}
 \theoremstyle{dotlessrem}
\newtheorem*{rem}{Remark}

\renewenvironment{abstract}
 {\small
  \begin{center}
  \bfseries \abstractname\vspace{-.5em}\vspace{0pt}
  \end{center}
  \list{}{
    \setlength{\leftmargin}{1.5cm}%
    \setlength{\rightmargin}{\leftmargin}%
  }%
  \item\relax}
 {\endlist}

\begin{document}
\begin{center} \Large Spectral asymptotics on the Hanoi attractor\\
\large Elias Hauser\footnote{Institute of Stochastics and Applications, University of Stuttgart, Pfaffenwaldring 57, 70569 Stuttgart, Germany, E-mail: elias.hauser@mathematik.uni-stuttgart.de}\end{center}
\begin{abstract}
The Hanoi attractor (or Stretched Sierpi\'nski Gasket) is an example of a non self similar fractal that still exhibits a lot of symmetry. The existence of various symmetric resistance forms on the Hanoi attractor was shown in 2016 by Alonso-Ruiz, Freiberg and Kigami \cite{afk17}. To get self adjoint operators from these resistance forms we have to choose a locally finite measure. The goal of this paper is to calculate the leading term for the asymptotics of the eigenvalue counting function from these operators.
\end{abstract}
\section{Introduction}\label{chap0}
The goal in this paper is to calculate the the leading term in the asymptotics of the eigenvalue counting function for a class of operators on the Hanoi attractor, which is a non self similar set. \\
Spectral asymptotics is an important tool in physics, for example to calculate how heat or waves propagate through media. For bounded domains $\Omega\subset \mathbb{R}^n$ the Dirichlet Laplacian has non negative discrete spectrum. The eigenvalue counting function $N_D^\Omega(x)$ has the following asymptotic behaviour
\begin{align}
N_D^\Omega(x)=\frac{\tau_n}{(2\pi)^n}\operatorname{Vol}_n(\Omega) x^{\frac n2}+ o(x^{\frac n2})
\end{align}
where $\tau_n$ is the volume of the unit ball in $\mathbb{R}^n$. This result is originally due to Weyl \cite{we11}. In the 70s the interest in fractals grew thanks to Mandelbrot. Their fine structure can be usefull in a better modelling of many naturally occuring structures and processes. To be able to do analysis on fractals one needs a Laplacian. One way to construct this operator is called the analytical approach which is due to Kigami. In \cite{kig89} he defined the Laplacian on the Sierpi\'nski Gasket as the limit of renormalized discrete Laplacians on approximating graphs. The intuitiv generalization of (1) to Laplacians on fractals would be
\begin{align}
N_D^F(x)=C_d\mathcal{H}^d(F)x^{\frac d2}+o(x^{\frac d2})
\end{align}
where $d=\dim_H(F)$ is the Hausdorff-Dimension of $F$, $\mathcal{H}^d(F)$ the $d$-dimensional Hausdorff-measure of $F$ and $C_d$ a constant that only depends on $d$. This was conjectured by Berry in \cite{ber1} and \cite{ber2}. However this turned out to be false. Shima \cite{shim} and Fukushima-Shima \cite{fushim} calculated the eigenvalues of the Laplacian on the Sierpi\'nski Gasket via the eigenvalue decimation method. The leading term in the asymptotics is $\frac 12 \frac{\ln 9}{\ln5}$ which does not coincide with the Hausdorff-Dimension. Another discovery was, that there is no constant in the asymptotic behaviour in front of the leading term but periodic behaviour. Later Kigami and Lapidus calculated the leading term for a class of fractals in \cite{kl93}, namely p.c.f. self similar fractals. Another generalization is from Kajino \cite{kaj10}, where he calculated the leading term for self similiar fractals in general.

The Hanoi attractor is a non self similar set, that still exhibits a lot of symmetry. In \cite{af12} the set was analyzed geometrically by Alonso-Ruiz and Freiberg by calculating its Hausdorff-Dimension. This set got its name by the connection to the game "The towers of Hanoi", which can also be found in \cite{af12}.

Let $p_1,p_2,p_3$ be the vertex points of a equilateral triangle with side length 1 and for $\alpha \in(0,1)$
\begin{align*}
G_i(x)&:=\frac{1-\alpha}2(x-p_i) +p_i, \ i\in \{1,2,3\}=:\mathcal{A}	\\[0.1cm]
e_1&:=\{\lambda G_2(p_3)+(1-\lambda)G_2(p_3) \ : \lambda \in  (0,1)\}\quad e_2,e_3 \ \text{analog}
\end{align*}
Then there exists a unique compact set $K_\alpha$ with
\begin{align*}
K_\alpha = G_1(K_\alpha) \cup G_2(K_\alpha) \cup G_3(K_\alpha) \cup e_1\cup e_2\cup e_3
\end{align*}
This set is called \textit{Hanoi attractor} or \textit{Stretched Sierpi\'nski Gasket (SSG)} since the contraction ratios are smaller than the ones of the Sierpi\'nski Gasket and the gaps are filled with one-dimensional lines. Also define $\Sigma_\alpha$ as the unique solution to
\begin{align*}
\Sigma_\alpha=G_1(\Sigma_\alpha)\cup G_2(\Sigma_\alpha)\cup G_3(\Sigma_\alpha)
\end{align*}

The sets $K_\alpha$ for $\alpha\in(0,1)$ are pairwise homeomorphic \cite[Prop. 2.4]{afk17} and since the resistance forms only depend on the topology of $K_\alpha$ we can omit the parameter $\alpha$ in the notation.
\begin{figure}[H]
\centering
\includegraphics[scale=0.1]{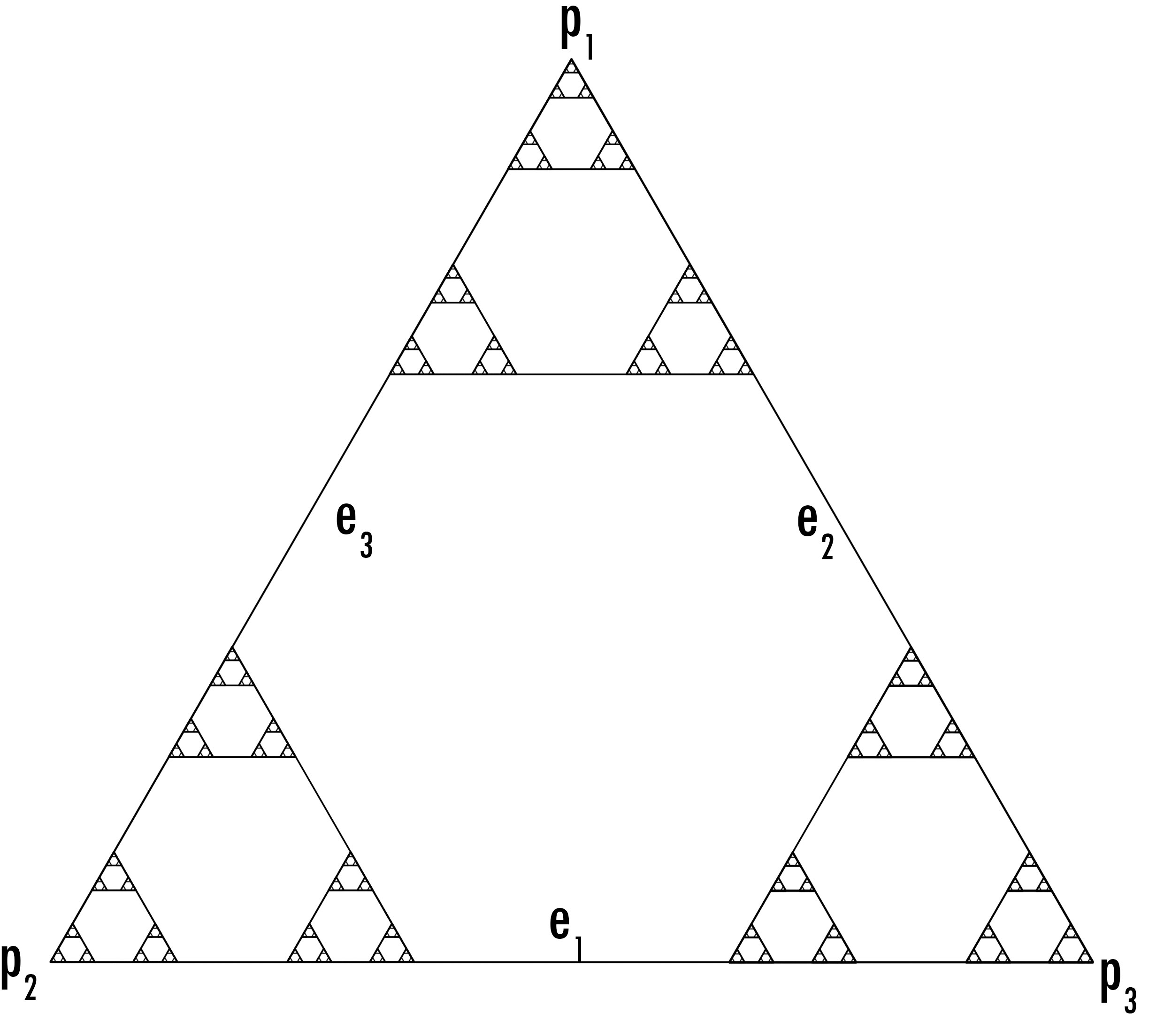}
\caption{Hanoi attractor}
\label{hanoi}
\end{figure}
We need to introduce some common notation. Let $\mathcal{A}:=\{1,2,3\}$, for $w\in\mathcal{A}^m$ with $m\in\mathbb{N}_0$:
\begin{itemize} 
\item $G_w=G_{w_1}\circ\ldots \circ G_{w_m}$ (with $G_w=\operatorname{id}$ for the empty word $w\in\mathcal{A}^0$)
\item $V_0:=\{p_1,p_2,p_3\}$,  $V_m:=\bigcup_{w\in \mathcal{A}^m}G_w(V_0)$ 
\item $e_i^w:=G_w(e_i)$
\item $K_w:=G_w(K)$,  $K_m:=\bigcup_{w\in\mathcal{A}^m}K_w$
\item $J_m:=K\backslash K_m$
\item $\Sigma_w:=G_w(\Sigma)$, $\Sigma_m:=\bigcup_{w\in\mathcal{A}^m}\Sigma_w$
\end{itemize}
We refer to $\Sigma$ as the fractal part and to $J=K\backslash \Sigma$ as the line part of $K$.\\[0.2cm]

There are two prior works concerning spectral asymptotics on the Hanoi attractor. The first is also by Alonso-Ruiz and Freiberg \cite{af13}. There they constructed a Dirichlet form and calculated the leading term in the asymptotics of the eigenvalue counting function of the associated operator. This leading term turns out to be the same value as for the Sierpi\'nski Gasket, namely $\frac{\ln 3}{\ln 5}$. The resistance form used corresponds to one coming from a fixed sequence of matching pairs (see chapter \ref{chap1})and the measure is the sum of the normalized Hausdorff-measure on the self similar part and a scaled lebesgue measure on the line part (compare to chapter \ref{chap4}). However the scaling parameter is chosen in such a way, that the influence of these edges is not too big.

Another work is by Alonso-Ruiz, Kelleher and Teplyaev \cite{akt16}. The approach in this work is by the use of quantum graphs. The Hanoi attractor is viewed as a so called fractal quantum graph. The measure used on the one-dimensional edges is more general than the one in \cite{af13}, however there is no mass on the higher dimensional fractal part. Therefore the calculated leading term in the asymptotics of the eigenvalue counting function turns out to be smaller than $\frac{\ln 3}{\ln 5}$.

In this work we combine the two works and generalize them to a class of resistance forms. These resistance forms were introduced by Alonso-Ruiz, Freiberg and Kigami in \cite{afk17}. These resistance forms consist of two parts. One belongs to the higher dimensional part of the Hanoi attractor and is very similar to the resistance form of the Sierpi\'nski Gasket. The other one belongs to the one-dimensional edges. As mentioned the choice of the resistances is not unique. Therefore there exists a whole class of resistance forms on the Hanoi attractor. In \cite{afk17} the authors treat the so called \textit{completely symmetric} ones, that exhibit the intuitiv symmetries. This term is defined in the work. It is also shown that each of the \textit{completely symmetric resistance forms} is of the discussed art.

In the current work we use these resistance forms and choose a suitable measure to get regular Dirichlet forms and thus self adjoint operators with non negative discrete spectrum. This spectrum can be analyzed in terms of the eigenvalue counting function and its asymptotic behaviour. 

This paper is organized as follows. In chapter \ref{chap1} the construction of the resistance forms from \cite{afk17} is briefly discussed. To be able to do the calculations we need to set some conditions on the resistances. These conditions are introduced in chapter \ref{chap2} following some important estimates for the resistance forms. In chapter \ref{chap3} the Hausdorff-Dimension of the Hanoi attractor is calculated with respect to resistance metric coming from a resistance form that fulfills the conditions. This value is more usefull for the analysis of a set, than the one calculated with respect to the euclidean metric. In chapter \ref{chap4} the measures that are used are introduced. After stating the results of this work in chapter \ref{chap5}, the proofs follow in chapter \ref{chap6}. This paper closes in chapter \ref{chap7} with some generalizations on the conditions.

\section{Recapitulation of Hanoi attractor and resistance forms}\label{chap1}
To be able to study analysis on the Hanoi attractor we need to introduce a resistance form on $K$. A definition of resistance forms can be found in \cite{kig12}. The choice of the resistance form is not unique and so we get different operators and different spectral asymptotics. The construction of these resistance forms was carried out in \cite{afk17}. The following paragraph will include a brief recapitulation of this construction.
\begin{figure}[H]
\centering
\includegraphics[scale=0.15]{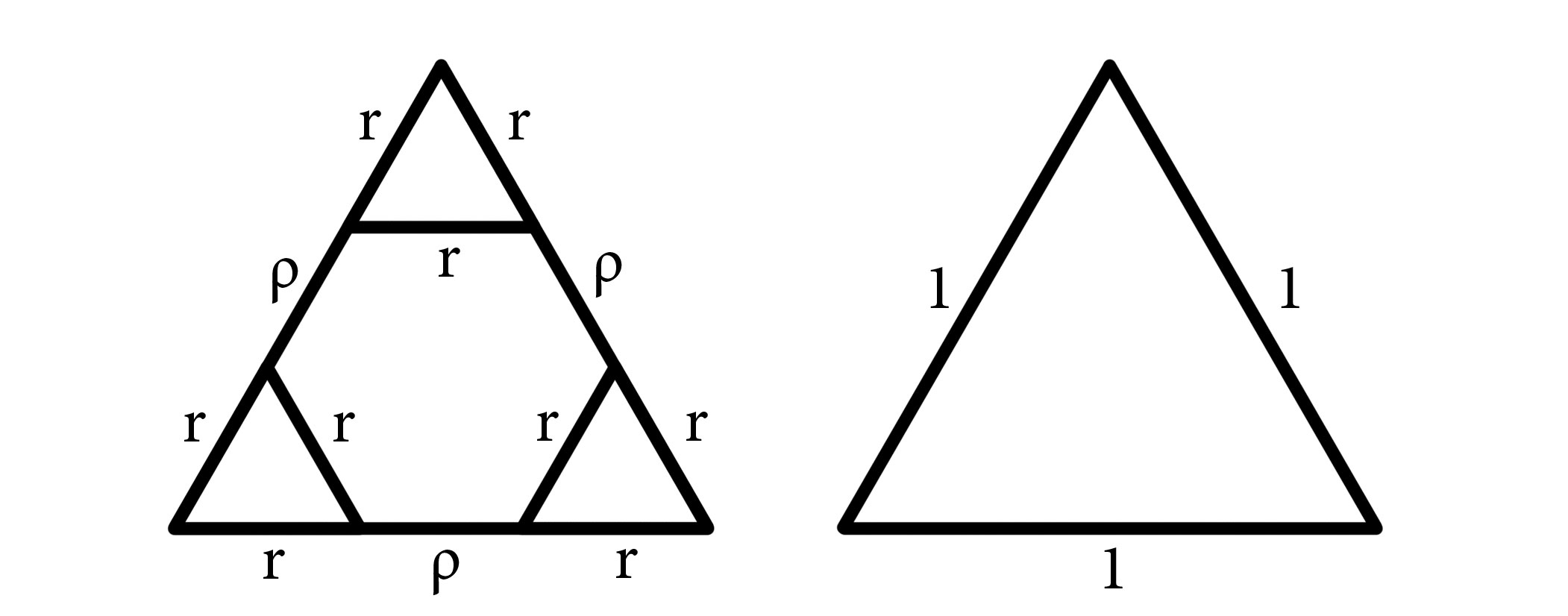}
\caption{Resistances}
\label{resis}
\end{figure}
In Figure \ref{resis} you can see the first graph approximation of $K$ beside the graph that just contains the vertices $p_1,p_2$ and $p_3$. Due to symmetry we want to have the resistances on the smaller triangles all equal $r$ and also all equal $\rho$ on the edges adjoining them. This electric network should be equivalent to the one on the right with all resistances equal $1$. A quick calculation with the help of the $\Delta-Y$-transformation leads to 
\begin{align*}
\frac 53 r+\rho=1
\end{align*}
Such a pair $(r,\rho)$ is then called a matching pair. In the next graph approximation the smaller triangles get divided further in the same fashion. 
\begin{figure}[H]
\centering
\includegraphics[scale=0.15]{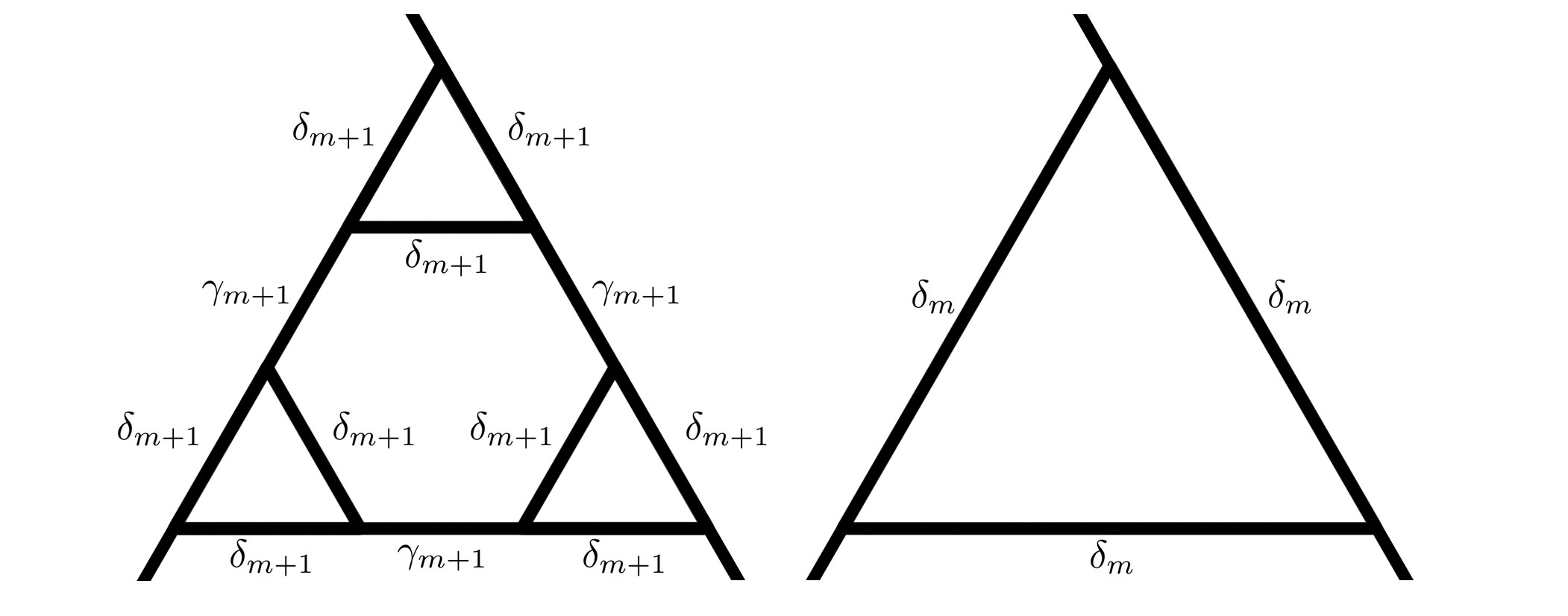}
\caption{Resistances in the $m+1$ graph approximation}
\label{resis2}
\end{figure}
In general in the $m+1$ graph approximation the left triangle in Figure \ref{resis2} has to be equivalent to the right one with all resistances $\delta_m$. 
The same calculation as for the first graph approximation shows, that it has to hold that 
\begin{align*}
\delta_{m+1}=\delta_m\cdot r_{m+1} \hspace*{0.5cm} \text{and} \hspace*{0.5cm} \gamma_{m+1}=\delta_m\cdot\rho_{m+1}
\end{align*}
with a matching pair $(r_{m+1},\rho_{m+1})$ i.e. $\frac 53 r_{m+1}+ \rho_{m+1}=1$. Notice that the resistances of the edges connecting adjoining cells from the previous graph approximations do not change.\\[.2cm]
We get for the $m$-th graph approximation, that
\begin{align*}
\delta_m=r_1\cdots r_m\hspace*{0.5cm} \text{and} \hspace*{0.5cm} \gamma_m=r_1\cdots r_{m-1} \rho_m
\end{align*}
with $\frac 53 r_i+\rho_i=1$ for all $i$. Such a sequence $\mathcal{R}=(r_i,\rho_i)_{i\geq 1}$ of matching pairs is also called a compatible sequence because each of those sequences will lead to a resistance form on $K$. 

With these resistances we can define a quadratic form. This form will consist of two parts. One part is very similar to the usual resistance form on the Sierpi\'nski Gasket. For $u\in \ell(K)$ define
\begin{align*}
Q_0^\Sigma(u,u)&:=(u(p_1)-u(p_2))^2+(u(p_2)-u(p_3))^2+(u(p_3)-u(p_1))^2\\
Q_m^\Sigma(u,u)&:=\sum_{w\in\mathcal{A}^m} Q_0^\Sigma(u\circ G_w,u\circ G_w)\\[0.2cm]
\mathcal{E}^\Sigma_{\mathcal{R}}(u,u)&:=\lim_{m\rightarrow \infty} \frac 1{\delta_m}Q_m^\Sigma (u,u)
\end{align*}
	
However this form ignores the adjoining edges of the Hanoi attractor. To get a form on the whole $K$ we need a second part. This can be achieved with the usual one-dimensional Dirichlet energy summed over all edges. With $\xi_{e^w_i}(t)=(1-t)(e^w_i)_-+t(e^w_i)^+$, $t\in (0,1)$, where $(e^w_i)_-$ and $(e^w_i)^+$ are the endpoints of $e^w_i$,  define
\begin{align*}
\mathcal{D}^I_k(u,u)&:=\sum_{ \begin{array}{c} w\in\mathcal{A}^{k-1}\\ i\in\{1,2,3\}\end{array} } \int_0^1 \left(\frac{d(u\circ \xi_{e^w_i})}{dx}\right)^2 dx\\[0.1cm]
\mathcal{E}^I_{\mathcal{R}}(u,u)&:=\sum_{k=1}^\infty \frac 1{\gamma_k} \mathcal{D}^I_k(u,u)
\end{align*}
Now we define the sum of the two parts as our final quadratic form:
\begin{align*}
\mathcal{E}_{\mathcal{R}}(u,u):=\mathcal{E}^\Sigma_{\mathcal{R}}(u,u)+ \mathcal{E}^I_{\mathcal{R}}(u,u)
\end{align*}
The form $\mathcal{E}_{\mathcal{R}}$ is defined on
\begin{align*}
\mathcal{F}_{\mathcal{R}}=\left\{u\in C(K) : \mathcal{E}_{\mathcal{R}}(u,u)<\infty, \  u|_{e^w_i}\in H^1(e_i^w), \forall w\in\mathcal{A}^m, m\in\mathbb{N}_0\right\}
\end{align*}
where $H^1(e_i^w)=\{u\in \ell(e_i^w), u\circ \xi_{e_i^w}\in H^1(0,1)\}$. One of the main results of \cite{afk17} is that for a sequence of matching pairs $\mathcal{R}=(r_i,\rho_i)_{i\geq 1}$ the form $(\mathcal{E}_{\mathcal{R}},\mathcal{F}_{\mathcal{R}})$ is indeed a regular resistance form.

The construction of these resistance forms can be studied in much greater detail in \cite{afk17}.

\section{Conditions and estimates of the resistance forms}\label{chap2}
The goal is to study the Hanoi attractor. One way to analyse a set is to study its geometric properties. Probably the most significant geometric value is the Hausdorff-Dimension. To calculate it, we have to choose a metric. Since the Hanoi attractor can be embedded in the $\mathbb{R}^2$ we could choose the euclidean metric. This value was calculated in \cite{af12}. However it depends on $\alpha$. Since the resistance forms do not depend on $\alpha$ we want to choose another metric that shares this characteristic. The resistance metric only depends on the resistances, therefore only on the sequence of matching pairs $\mathcal{R}$, and furthermore it reflects the analysis of the set much better.\\
We also want to study the analysis of the Hanoi attractor, in particular the spectral asymptotics.

In the self similar case, these calculations can be done with the help of the self similar scaling properties of the resistance forms. Since we are not in the self similar case we do not have these tools. We are able to get some estimates on the resistance forms, but to be able to do these calculations we need to introduce some conditions on the sequences of matching pairs.\\

\textbf{Conditions} on the sequences of matching pairs $\mathcal{R}=(r_i,\rho_i)_{i\geq 1}$:\\[0.2cm]
$\frac 13\leq  r<\frac 35$:
\begin{align*}
\sum_{i=1}^\infty | r- r_i|<\infty 
\end{align*}
For $r=\frac 35$
\begin{align*}
\sum_{i=1}^\infty \rho_i<\infty , \ \text{ and } (r_i)_i \text{ mon. increasing}
\end{align*}
These conditions ensure, that $r_i\rightarrow r$ fast enough. For $r=\frac 35$ we need the additional condition, that the convergence is monotone.

With these conditions we get estimates for the rescaling of the resistance forms. However the next lemma holds for all sequences of matching pairs where $(r_i)_{i\geq 1}$ converges.

\begin{lem} If $r_i\rightarrow r$, then
\begin{align*} \mathcal{E}_\mathcal{R}^\Sigma (u,u)= r^{-1}\sum_{i=1}^3 \mathcal{E}_\mathcal{R}^\Sigma (u\circ G_i,u\circ G_i)
\end{align*}
\label{lem21}
\end{lem}
\begin{proof}
\begin{align*} \mathcal{E}_\mathcal{R}^\Sigma(u,u)&=\lim_{m\rightarrow \infty} \frac 1{\delta_m} Q_m^\Sigma(u,u)\\
&=\lim_{m\rightarrow \infty} \frac 1{\delta_m} \sum_{i=1}^3 Q_{m-1}^\Sigma(u\circ G_i,u\circ G_i)\\
&=\lim_{m\rightarrow \infty} \frac{\delta_{m-1}}{\delta_m}\frac 1{\delta_{m-1}} \sum_{i=1}^3 Q_{m-1}^\Sigma(u\circ G_i,u\circ G_i)\\
&=\lim_{m\rightarrow \infty} \frac 1{r_m}\frac 1{\delta_{m-1}} \sum_{i=1}^3 Q_{m-1}^\Sigma(u\circ G_i,u\circ G_i)\\
&=r^{-1} \sum_{i=1}^3 \mathcal{E}_\mathcal{R}^\Sigma(u\circ G_i, u \circ G_i)
\end{align*}
\end{proof}
\begin{cor} If $r_i\rightarrow r$, then
\begin{align*} \mathcal{E}_\mathcal{R}^\Sigma (u,u)= r^{-m}\sum_{w\in\{1,2,3\}^m} \mathcal{E}_\mathcal{R}^\Sigma (u\circ G_w,u\circ G_w)
\end{align*}
\end{cor}
Now for the line part of the resistance form. Here we don't have an equality, but we get upper and lower estimates.
The conditions on the sequences of matching pairs ensure, that
\begin{align*}
R^\star:=\prod_{i=1}^\infty r^{-1}r_i \in(0,\infty)
\end{align*}
Therefore $(a_m)_{m\geq 1}:=(\prod_{i=1}^mr^{-1}r_i)_{m\geq 1}$ converges in $(0,\infty)$ and is thus bounded from above and below. There exist $\kappa_1,\kappa_2\in(0,\infty)$ such that
\begin{align*}
\kappa_1\leq &\ a_m\leq \kappa_2 ,\ \forall m
\end{align*}
and thus:
\begin{lem} \begin{align*}
\kappa_1 r^m \leq &\ \delta_m \leq \kappa_2r^m, \ \forall m
\end{align*}\label{lem22}
\end{lem}
\noindent With this, we get estimates on the line part of the resistance form.
\begin{align*} \mathcal{E}_\mathcal{R}^I( u, u)&=\sum\limits_{k=1}^\infty \frac 1{\gamma_k} \mathcal{D}^I_k( u,u) \\
&\geq \sum\limits_{k=m+1}^\infty \frac 1{\gamma_k}\mathcal{D}^I_k(u,u)\\
&= \sum_{w\in\mathcal{A}^m}\sum\limits_{k=1}^\infty \frac 1{\tilde \gamma_k}\mathcal{D}^I_k(u\circ G_w,u\circ G_w)
\end{align*}
with
\begin{align*}\tilde\gamma_k=\gamma_{m+k}=r_1\cdots r_{m+(k-1)}\rho_{m+k}\end{align*}

For these new resistance factors for the edges we have two cases: \\
For $r<\frac 35$ the sequence $(\rho_k)_{k\geq 1}$ converges to $1-\frac 53r>0$ and is thus bounded from below by a constant $\kappa_3>0$.
\begin{align*} \tilde\gamma_k&=\frac{\delta_{m+k-1}}{\delta_{k-1}} \frac{\rho_{m+k}}{\rho_k}\gamma_k\\
&\leq \frac{\kappa_2}{\kappa_1}r^m \frac 1\kappa_3 \gamma_k\\
&=\tilde{\mathcal{K}}  r^m \gamma_k
\end{align*}
For $r=\frac 35$ the monotonicity gives the same estimate with $\tilde{\mathcal{K}} =1$.
\begin{align*}
\Rightarrow \mathcal{E}_\mathcal{R}^I(u,u)&\geq \frac 1{\tilde{\mathcal{K}}}  r^{-m} \sum_{w\in\mathcal{A}^m} \sum_{k=1}^\infty \frac 1{\gamma_k}\mathcal{D}^I_k(u\circ G_w,u\circ G_w)\\
&=\frac 1{\tilde{\mathcal{K}}}  r^{-m} \sum_{w\in\mathcal{A}^m}\mathcal{E}^I_\mathcal{R}(u\circ G_w,u\circ G_w)
\end{align*}
and therefore with $K:=\min\{1,1/\tilde{\mathcal{K}}\}$\\[0.2cm]
\begin{lem}\begin{align*}
\mathcal{E}_\mathcal{R}(u,u)\geq \mathcal{K}  r^{-m} \sum_{w\in\mathcal{A}^m}\mathcal{E}_\mathcal{R}(u\circ G_w,u\circ G_w)
\end{align*}\label{lem23}
\end{lem}

\section{Hausdorff-Dimension in resistance metric} \label{chap3}
The topology of $K$ is the same with either the euclidean or the resistance metric $R_\mathcal{R}$ coming from one of the completely symmetric resistance forms described before \cite{afk17}. \\
That means the closure of $J=K\backslash \Sigma =\bigcup_{w\in\mathcal{A}^m, m\in \mathbb{N}_0, i\in\{1,2,3\}}e^i_w$ with respect to $R_\mathcal{R}$ is the same as with the euclidean metric.\\
Considering $K=\Sigma\cup J$, we have $\dim_{H,R_\mathcal{R}}K=\max\{\dim_{H,R_\mathcal{R}}\Sigma, \dim_{H,R_\mathcal{R}} J\}$ where $\dim_{H,R_\mathcal{R}}$ denotes the Hausdorff-Dimension calculated with respect to the the resistance metric $R_\mathcal{R}$.\\[0.2cm]
\begin{lem}$\dim_{H,R_\mathcal{R}} J\leq 1$\label{lem31}\end{lem}
\begin{proof}
To see this, we show that $\dim_{H,R_\mathcal{R}}(e^i_w)\leq 1 $ for each $i\in\{1,2,3\}$ and $w\in\mathcal{A}^m$. The result follows with the $\sigma$-stability of the Hausdorff-Dimension.\\

$d_e$ denotes the diameter with respect to the euclidean metric.
\begin{align*}
&x,y\in e^i_w \Rightarrow R_\mathcal{R}(x,y)\leq \gamma_{|w|+1}\cdot |x-y|\cdot  d_e(e^i_w)^{-1}\\
&\mathcal{H}^s_{\delta,R_\mathcal{R}}(e^i_w)=\inf\left\{\sum d_{R_\mathcal{R}}(U_i)^s | \{U_i\}_i \ \delta\text{-covering of }e^i_w\right\}, \\
&\text{with } d_{R_\mathcal{R}}(U_i)=\sup\limits_{x,y\in U_i} R_\mathcal{R}(x,y)\leq \frac{\gamma_{|w|+1}}{d_e(e_w^i)}\sup\limits_{x,y\in U_i}|x-y|=\frac{\gamma_{|w|+1}}{d_e(e_w^i)} d_{e}(U_i)\\
& \Rightarrow \mathcal{H}^s_{\delta,R_\mathcal{R}}(e^i_w)\leq \left(\frac{\gamma_{|w|+1}}{d_e(e_w^i)}\right)^s\cdot  \mathcal{H}^s_{\delta,e}(e^i_w)\\
&\Rightarrow \mathcal{H}^s_{R_\mathcal{R}}(e^i_w)\leq \left(\frac{\gamma_{|w|+1}}{d_e(e_w^i)}\right)^s\cdot \mathcal{H}^s_{e}(e^i_w)\\
&\Rightarrow \mathcal{H}^1_{R_\mathcal{R}}(e^i_w)\leq \frac{\gamma_{|w|+1}}{d_e(e_w^i)} \mathcal{H}^1_{e}(e^i_w)<\infty
\end{align*}
And therefore $\dim_{H,R_\mathcal{R}}(e^i_w)\leq 1$. \end{proof}
\begin{lem} Let $\mathcal{R}=(r_i,\rho_i)_{i\geq 1}$ be a sequence of matching pairs that fulfills the conditions, then there exists $c>0$, such that
\begin{align*}
d_{R_\mathcal{R}}(\Sigma_w)\leq c\cdot r^m, \ \text{ for all } w \in \{1,2,3\}^m , \ m\in\mathbb{N}
\end{align*}
\label{lem32}\end{lem}
\begin{proof}
Let $a,b\in \Sigma_w$, then there exist $x,y\in \Sigma$ with $G_w(x)=a$ and $G_w(y)=b$. 
\begin{align*} \min\{\mathcal{E}_\mathcal{R}(u,u): u\in\mathcal{F}_\mathcal{R}, u(a)=0, u(b)=1\}=R_\mathcal{R}(a,b)^{-1} \end{align*}
Let $\tilde u$ be the function for which the minimum is attained. Then $\tilde u\circ G_w \in \mathcal{F}_\mathcal{R}$ due to Lemma \ref{lem23} and $(\tilde u\circ G_w)(x)=0$ and $(\tilde u\circ G_w)(y)=1$, therefore it is one of the functions for which the energy is minimized to calculate $R_\mathcal{R}(x,y)$.\\
From Lemma \ref{lem23} we have
\begin{align*}
\mathcal{E}_\mathcal{R}(u,u)\geq \mathcal{K}  r^{-m} \sum_{\tilde w\in\mathcal{A}^m}\mathcal{E}_\mathcal{R}(u\circ G_{\tilde w},u\circ G_{\tilde w})
\end{align*}
and if we loose all $m$-cells but $K_w$ we have
\begin{align*}
\mathcal{E}_\mathcal{R}(u,u)\geq \mathcal{K}  r^{-m} \mathcal{E}_\mathcal{R}(u\circ G_{w},u\circ G_{ w})
\end{align*}
\begin{align*}
\Rightarrow R_\mathcal{R}(x,y)^{-1} \leq \mathcal{E}_\mathcal{R}(\tilde u \circ G_w, \tilde u \circ G_w)\leq \mathcal{K}^{-1}r^m\mathcal{E}_\mathcal{R}(\tilde u,\tilde u)=\mathcal{K}^{-1}r^m R_\mathcal{R}(a,b)^{-1}\end{align*}
\begin{align*}
\Rightarrow R_\mathcal{R}(a,b)\leq\mathcal{K}^{-1} r^m R_\mathcal{R}(x,y)\leq \mathcal{K}^{-1} r^m \sup\limits_{\tilde x,\tilde y\in \Sigma} R_\mathcal{R}(\tilde x,\tilde y) \leq cr^m
\end{align*}
This holds for all $a,b\in \Sigma_w$, therefore $d_{R_\mathcal{R}}(\Sigma_w)\leq c r^m$\end{proof}

\begin{lem} Let $\mathcal{R}=(r_i,\rho_i)_{i\geq 1}$ be a sequence of matching pairs that fulfills the conditions, then it holds for all $x\in\Sigma$, that there is a constant $c>0$, such that
\begin{align*} \#\left\{w \in \{1,2,3\}^m \ |\ R_\mathcal{R}(x,\Sigma_w) \leq c r^m \right\}\leq 4, \qquad \text{for all } m\in\mathbb{N} 
\end{align*}
\label{lem33}\end{lem}
\begin{proof}
\begin{align*} R_\mathcal{R}(x,y)^{-1} = \min\{ \mathcal{E}_\mathcal{R}(u,u) \ : \ u\in\mathcal{F}_\mathcal{R}, \ u(x)=0, u(y)=1\}\end{align*}
For a fixed $u\in \mathcal{F}_\mathcal{R}$ with $u(x)=0$ and $u(y)=1$ we have
\begin{align*} R_\mathcal{R}(x,y)\geq \frac 1{\mathcal{E}_\mathcal{R}(u,u)}
\end{align*}
We are looking for such a $u$, so that this estimate is good enough. Let $w \in \{1,2,3\}^m$, $y \in \Sigma_w$ and $x\in \Sigma\backslash \Sigma_w$.\\
Define $u_m$ on $V_m$ as 
\begin{align*} u_m(l)=1&, \ \text{for all } l \in G_w(V_0) \text{ and any } l\in V_m \text{ adjoining } G_w(V_0)\\
u_m(l)=0&, \text{otherwise}
\end{align*}
\begin{figure}[H]
\centering
\includegraphics[scale=0.1]{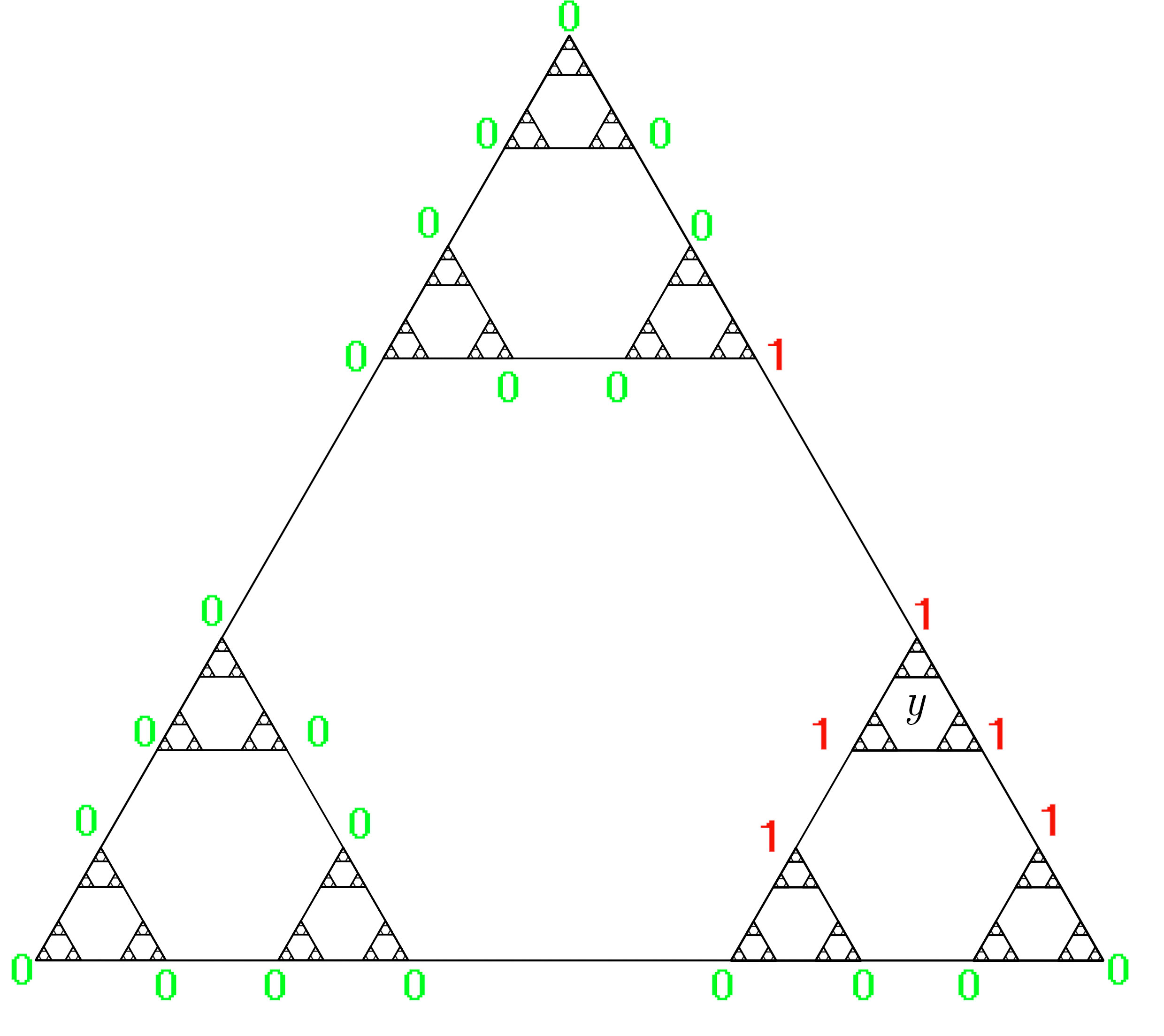}
\caption{Construction of $u_m$}
\label{hausdim}
\end{figure}
\noindent In figure \ref{hausdim} the construction of this function is illustrated where $y$ lies anywhere in the $2$-cell which is marked with "y".\\ 

Then define $u\in \mathcal{F}_\mathcal{R}$ as the harmonic extension of $u_m$. The extension $u$ is constant $1$ on $K_w$ (and therefore $\Sigma_w$) and constant $0$ in all but at most three $m$-cells differing from $K_w$. For $x$ in these $m$-cells where $u$ is constant $0$ we can use this function to get an estimate of $R_\mathcal{R}(x,y)$ for all $y\in \Sigma_w$. 
\begin{align*} \mathcal{E}_\mathcal{R}(u,u) \leq 3\cdot 2 \frac 1 { r_1\cdot r_m}=6\frac 1{\delta_m}
\end{align*}
From Lemma \ref{lem22} we get
\begin{align*} &\Rightarrow \mathcal{E}_\mathcal{R}(u,u) \leq  \frac 6{\kappa_1} r^{-m} \\
&\Rightarrow R_\mathcal{R}(x,y) \geq \frac{\kappa_1}6 r^m 
\end{align*}
For fixed $x$ there are at most four $m$-cells for which this construction does not work. We therefore have the desired result.\end{proof}
\begin{theo} Let $\mathcal{R}=(r_i,\rho_i)_{i\geq 1}$ be a sequence that fulfills the conditions, then
\begin{align*}
\dim_{H,R_\mathcal{R}}(K)=\frac{\ln 3}{-\ln r}
\end{align*}
\label{theo34}\end{theo}
\begin{proof}
From Lemma \ref{lem32} and \ref{lem33} it follows with \cite[Theo. 2.4]{kig95} that $$\dim_{H,R_\mathcal{R}}(\Sigma)=\frac{\ln 3}{-\ln r}$$ From Lemma \ref{lem31} we know, that $$\dim_{H,R_\mathcal{R}}(J)\leq 1\leq \dim_{H,R_\mathcal{R}}(\Sigma) \Rightarrow \dim_{H,R_\mathcal{R}}(K)=\dim_{H,R_\mathcal{R}}(\Sigma)$$ \end{proof}
\begin{rem} For $r=\frac 35$ this is the same value as for the Sierpi\'nski Gasket with the usual resistance form. \end{rem}
\section{Measures and operators}\label{chap4}
To get Dirichlet forms and Laplacians on $K$ we need a measure $\mu$ on $K$. This measure has to fulfill a few requirements. It has to be supported on $K$ and has to be locally finite (in particular finite due to the compactness of $K$). \\
We want to describe the measure on $K$ as the sum of a fractal- and a line-part in accordance to the geometric appearance of $K$.

It is clear how the fractal part of $\mu$ has to be choosen. If we distribute the mass equally on all m-cells of $K$ we get the normalized Hausdorff-measure on $K$:
\begin{align*}
\mu_f(K_w)=\mu_f(\Sigma_w)=\frac 1{3^{|\omega|}}
\end{align*} 
However this measure is too rough to measure the line parts. These parts of $K$ are just ignored by $\mu_f$. That means we need another measure $\mu_l$ which is able to measure $J$. How should the mass of the line segments scale to get an appropriate measure on $J$? For $a,\beta>0$:
\begin{align*}
\mu_l(e_i)&=a, \forall i\in\{1,2,3\}\\
\mu_l(e^w_i)&=a\beta^{m}, \forall i \in\{1,2,3\}, w\in\mathcal{A}^{m}
\end{align*}
For $\mu_l(J)<\infty$ it has to hold, that $\beta\in(0,\tfrac 13)$. 
\begin{align*}
\mu_l(J)&= 3a+3^2\beta a+3^3 \beta^2a+ \cdots\\
&=3a\sum_{k=0}^\infty (3\beta)^k\\
&=\frac{a}{\frac 13-\beta}\stackrel{!}{=}1\\[0.2cm]
\Rightarrow a&=\frac 13 -\beta
\end{align*}
If $\beta\rightarrow 0$ then more mass is distributed to the longer edges. For $\beta \rightarrow \frac 13 $ the mass is distributed more equally which displays the geometry better.\\[0.2cm]
We can now define the measure we will be using by the sum of those two parts:
\begin{align*}
\mu:=\frac 12(\mu_l+\mu_f)
\end{align*}
How does this measure scale for smaller cells? For the fractal part this is clear due to its definition:
\begin{align*}
\mu_f(K_w)=\frac 1{3^{|w|}}
\end{align*}
The measure $\mu_l$ on the line part exhibits another scaling. Since
\begin{align*}
1&=3^{|w|}\mu_l(K_w)+\sum_{\tilde w: |\tilde w|< |w|,i} \mu_l(e^i_{\tilde w})\\
&=3^{|w|}\mu_l(K_w)+3a\sum_{k=0}^{|w|-1}(3\beta)^k\\
&=3^{|w|}\mu_l(K_w)+(1-(3\beta)^{|w|})\\[0.2cm]
\Rightarrow \mu_l(K_w)&=\beta^{|w|}
\end{align*}
For $\mu$ we get the following estimates which will be usefull later on.
\begin{align*}
\beta^{|w|}\leq \mu(K_w)\leq \left(\frac 13\right)^{|w|}
\end{align*}
as well as for $B\in\mathcal{B}(\mathbb{R}^2)$
\begin{align*}
\beta^{|w|}\mu(B)\leq \mu(G_w(B))\leq \left(\frac 13\right)^{|w|}\mu(B)
\end{align*}

With these measures we can define Dirichlet forms and therefore operators on $L^2(K,\mu)$. Since $(K,R_\mathcal{R})$ is compact we have the following result with $\mathcal{D}_\mathcal{R}:=\overline{\mathcal{F}_\mathcal{R}\cap C_0(K)}^{\mathcal{E}_{_\mathcal{R},1}^{\frac 12}}=\mathcal{F}_\mathcal{R}$.\\
\begin{lem} $(\mathcal{E}_\mathcal{R},\mathcal{D}_\mathcal{R})$ is regular Dirichlet form on $L^2(K,\mu)$.\label{lem41}\end{lem}
\begin{proof}
From \cite[Theo. 9.4]{kig12} and \cite[Theo. 5.16]{afk17} it follows that $(\mathcal{E}_\mathcal{R},\mathcal{D}_\mathcal{R})$ is a regular Dirichlet form.\end{proof}
Introducing Dirichlet boundary conditions we get another Dirichlet form with $\mathcal{D}_\mathcal{R}^0:=\{u\in\mathcal{D}_\mathcal{R}: \ u|_{V_0}\equiv 0\}$.\\
\begin{lem} $(\mathcal{E}_\mathcal{R}|_{\mathcal{D}_\mathcal{R}^0\times\mathcal{D}_\mathcal{R}^0}, \mathcal{D}_\mathcal{R}^0)$ is a regular Dirichlet form on $L^2(K\backslash V_0,\mu)$.\label{lem42}\end{lem}
\begin{proof}
Since $V_0$ is finite it follows with \cite[Prop. 2.19]{kig03} that is a resistance form and therefore a Dirichlet form. It is regular since $\mathcal{D}_{\mathcal{R}}^0$ is dense in $C_0(K\backslash V_0)$ with respect to $||.||_\infty$.\end{proof}
We denote the associated self adjoint operators with dense domains by $-\Delta_N^{\mu,\mathcal{R}}$ resp. $-\Delta_D^{\mu,\mathcal{R}}$. \\
\begin{lem} $-\Delta_N^{\mu,\mathcal{R}}$ and $-\Delta_D^{\mu,\mathcal{R}}$ have discrete non negative spectrum.\label{lem43}\end{lem}
\begin{proof}
Since $(K,R_\mathcal{R})$ is compact it follows with \cite[Lemma 9.7]{kig12} that the inclusion map $\iota: \mathcal{D}_\mathcal{R}\hookrightarrow C(K)$ with the norms $\mathcal{E}_\mathcal{R}^{\frac 12}$ resp. $||\cdot ||_{\infty}$ is a compact operator. Since the inclusion map from $C(K)$ to $L^2(K,\mu)$ is continuous the inclusion from $\mathcal{D}_\mathcal{R}$ to $L^2(K,\mu)$ is a compact operator and therefore with \cite[Theo. 5 Chap. 10]{bs87} the spectrum of $-\Delta_N^{\mu,\mathcal{R}}$ is discrete and non-negative. Since $\mathcal{D}_\mathcal{R}^0\subset \mathcal{D}_\mathcal{R}$ the same follows for $-\Delta_D^{\mu,\mathcal{R}}$ by \cite[Theo. 4 Chap. 10]{bs87}. \end{proof}
\section{Results}\label{chap5}
Due to Lemma \ref{lem43} we can write the eigenvalues in nondecreasing order and study the eigenvalue counting functions. Denote by $\lambda_k^{N,\mu,\mathcal{R}}$ the $k$-th eigenvalue of $-\Delta_N^{\mu,\mathcal{R}}$ resp. $\lambda_k^{D,\mu,\mathcal{R}}$ for $-\Delta_D^{\mu,\mathcal{R}}$ with $k\geq 1$. Now define
\begin{align*}
N_N^{\mu,\mathcal{R}}(x):=\#\{k\geq 1: \lambda_k^{N,\mu,\mathcal{R}}\leq x\}\\
N_D^{\mu,\mathcal{R}}(x):=\#\{k\geq 1: \lambda_k^{D,\mu,\mathcal{R}}\leq x\}
\end{align*}
Since $\mathcal{D}_\mathcal{R}^0\subset \mathcal{D}_\mathcal{R}$ we immediately get 
\begin{align*}
N_D^{\mu,\mathcal{R}}(x)\leq N_N^{\mu,\mathcal{R}}(x), \ \forall x\geq 0
\end{align*}

We want to study the asymptotic behaviour of the eigenvalue counting functions. The next theorem is the main result of this work.\\
\begin{theo}
Let $\mathcal{R}:=(r_i,\rho_i)_{i\geq 1}$ be a sequence of matching pairs that fulfills the conditions. Then there exist constants $0<C_1,C_2<\infty$ and $x_0>0$, such that for all $x\geq x_0$:
\begin{align*}
C_1x^{\frac 12 d^{\mathcal{R}}_S(K)} \leq N_D^{\mu,\mathcal{R}}(x)\leq N_N^{\mu,\mathcal{R}}(x)\leq C_2x^{\frac 12 d^{\mathcal{R}}_S(K)}
\end{align*}
with 
\begin{align*}
d^{\mathcal{R}}_S(K)=\frac{\ln 9}{\ln3 - \ln r}
\end{align*}
\label{theo51}
\end{theo}
This value is the leading term in the spectral asymptotics. We will call it the spectral dimension of the Hanoi attractor. For $r=\frac 35$ this is the same value as for the Sierpi\'nski Gasket with the usual resistance form \cite{kl93}. Another observation is, that the measure scaling parameter $\beta$ of the line part does not show in the leading term. \\

We see that the choice of the sequence of matching pairs has a big influence on the analysis on $K$. \\
\begin{rem}
With $\dim_{H,R_\mathcal{R}}(K)=\frac{\ln 3}{-\ln r}$ it holds, that
\begin{align*}
d^{\mathcal{R}}_S(K)=\frac{2\dim_{H,R_\mathcal{R}}(K)}{\dim_{H,R_\mathcal{R}}(K)+1}
\end{align*}

This relation was shown to hold for p.c.f. self similar sets in \cite{kl93} and is now valid for a non self similar set.\end{rem}
\section{Proof of Theorem \ref{theo51}}\label{chap6}
Since $\mu$ and $\mathcal{R}$ are fixed throughout the whole proof we will omit them in the following whenever it is clear.

The main technique for the proof is the Dirichlet-Neumann bracketing as in \cite{kaj10}, where it was applied to self-similar sets. We split the proof in the upper and lower estimate.\\

\textbf{I: Upper estimate}\nopagebreak\\[0.2cm]
The upper bound is obtained by successively adding new Neumann boundary conditions at the points $V_m\backslash V_0$ thus making the domain bigger and therefore increasing the eigenvalue counting function. This is done by defining the domains
\begin{align*}
\mathcal{D}_{K_m}:&=\{ u \in L_2(K_m,\mu|_{K_m}) : \exists f\in \mathcal{D}: f|_{K_m}=u\}\\
&=\{u\in L_2(K,\mu): u|_{K_m^c}=0, \exists f\in \mathcal{D}: f|_{K_m}=u\}\\[.1cm]
\mathcal{D}_{J_m}:&=\{u\in L_2(J_m,\mu|_{J_m}) \exists f\in\mathcal{D}: f|_{J_m}=u\}\\
&=\{u\in L_2(K,\mu): u|_{J_m^c}=0 , \exists f\in \mathcal{D}: f|_{J_m}=u\}
\end{align*}
Considering $\mathcal{D}_{J_m}$ we see, that if we take $f$ to be harmonic on all of the $m$-cells, we get
\begin{align*}
\mathcal{D}_{J_m}=\bigoplus_{\begin{array}{c} w\in\mathcal{A}^n, n< m \\ i\in\{1,2,3\}\end{array} } H^1(e^i_w)  
\end{align*}
It is obvious, that $\mathcal{D}_{K_m}\perp \mathcal{D}_{J_m}$ and
\begin{align*}
\mathcal{D}\subset \mathcal{D}_{K_m}\oplus \mathcal{D}_{J_m}
\end{align*}
On this bigger domain we define the form $\tilde{\mathcal{E}}$ on $f=g+h$, with $g\in \mathcal{D}_{K_m}, h\in\mathcal{D}_{J_m}$
\begin{align*}
\tilde{\mathcal{E}}(f,f):=\mathcal{E}^\Sigma (g,g) + \sum_{k=m+1}^\infty \frac{1}{\gamma_k}\mathcal{D}_k^I(g,g) + \sum_{k=1}^m \frac{1}{\gamma_k}\mathcal{D}_k^I(h,h)
\end{align*}
and
\begin{align*}
\mathcal{E}_{K_m}(g,g)&=\mathcal{E}^\Sigma(g,g)+\sum_{k=m+1}^\infty\frac 1{\gamma_k}\mathcal{D}^I_k(g,g)\\
\mathcal{E}_{J_m}(h,h)&=\sum_{k=1}^m \frac 1{\gamma_k}\mathcal{D}^I_k(h,h)
\end{align*}
\begin{lem} $(\tilde{\mathcal{E}},\mathcal{D}_{K_m}\oplus\mathcal{D}_{J_m})$, $(\mathcal{E}_{K_m},\mathcal{D}_{K_m})$ and $(\mathcal{E}_{J_m},\mathcal{D}_{J_m})$ are regular Dirichlet forms with discrete non negative spectrum and $\tilde{\mathcal{E}}=\mathcal{E}_{K_m}\oplus\mathcal{E}_{J_m}$.\label{lem61}\end{lem}
\begin{proof}
$(\mathcal{E}_{J_m},\mathcal{D}_{J_m})$ is just the sum of scaled Dirichlet energys on one-dimensional edges, hence it is a regular Dirichlet form on $L_2(J_m,\mu|_{J_m})$ with discrete non-negative spectrum. Since $K_m$ is closed $(\mathcal{E}_{K_m},\mathcal{D}_{K_m})$ is a regular resistance form due to \cite[Theo. 8.4]{kig12} and hence a regular Dirichlet form on $L_2(K_m,\mu|_{K_m})$ with \cite[Theo. 9.4]{kig12}. Due to the same Theorem \cite[Theo. 8.4]{kig12} it follows that the associated resistance metric equals the restriction of $R$ to ${K_m}\times{K_m}$. Since $K_m$ is closed therefore $(K_m,R|_{K_m})$ is compact. The rest of the argument works like the proof of Lemma~\ref{lem43}.
The results for $\tilde{\mathcal{E}}$ follow immediately.
\end{proof}
We denote by $N(\mathcal{E},\mathcal{D},x)$ the eigenvalue counting function of a regular Dirichlet form which is the same as the one for the associated self adjoint operator.
For the eigenvalue counting functions of the mentioned forms this means:
\begin{align*}
N_N(x)\leq N(\mathcal{E}_{K_m},\mathcal{D}_{K_m},x)+ N(\mathcal{E}_{J_m},\mathcal{D}_{J_m},x), \ \forall x\geq 0
\end{align*}

The introduction of the Neumann boundary conditions at $V_m\backslash V_0$ leads to the decoupling of the $m$-cells and the edges adjoining them. Therefore the calculations can be done seperately.\\

\textbf{I.1: Fractal part $(\mathcal{E}_{K_m},\mathcal{D}_{K_m})$}\nopagebreak\\[0.2cm]
Define a set of measures on $K$ as follows
\begin{align*}
\mu^w:=\mu(K_w)^{-1}\mu\circ G_w
\end{align*}
$\mu^w$ is a measure on the whole $K$ but it just reflects the features of $\mu$ on $K_w$. We notice, that
\begin{align*}
\mu^w(K)=\mu(K_w)^{-1}\mu(K_w)=1, \ \forall w
\end{align*}
as well as
\begin{align*}
\int_Ku\circ G_w d\mu^w=\mu(K_w)^{-1}\int_{K_w}ud\mu 
\end{align*}
In the following proof we use the so called \textit{uniform poincar\'e inequality} (see \cite{kaj10}) for a $C_{PI}\in(0,\infty) $ and all $u\in \mathcal{D}$:
\begin{align*}
\mathcal{E}(u,u)\geq C_{PI}\int_K|u-\bar u^{\mu^w}|^2d\mu^w
\end{align*}
where $\bar u^\nu=\int_K ud\nu$. The constant $C_{PI}$ is independent of $w$. That this holds can be seen easily: Let $M:=\sup_{p,q\in K} R_\mathcal{R}(p,q)<\infty$. 
\begin{align*}
M\mathcal{E}(u,u)\geq R_\mathcal{R}(p,q)\mathcal{E}(u,u)&\geq |u(p)-u(q)|^2\\[0.2cm]
\Rightarrow \int_K\int_K M\mathcal{E}(u,u) d\mu^w(q)d\mu^w(p)&\geq \int_K\int_K |u(p)-u(q)|^2d\mu^w(q)d\mu^w(p)\\
&\geq \int_K \left( u(p)-\int_K u(q)d\mu^w(q)\right)^2d\mu^w(p)\\
&=\int_K|u(p)-\bar u^{\mu^w}|^2d\mu^w(p)\\[0.2cm]
\Rightarrow \mathcal{E}(u,u)\geq \frac 1{M\mu^w(K)^2}\int_K |u-\bar u^{\mu^w}|^2d\mu^w &=\frac 1M\int_K |u-\bar u^{\mu^w}|^2d\mu^w
\end{align*}
Since there are $3^m$ independent cells in $\mathcal{D}_{K_m}$ the $3^m$ first eigenvalues are all $0$, because the functions that are constant on each $m$-cell are in $\mathcal{D}_{K_m}$. We are interested in the first non-zero eigenvalue $\lambda^m_{3^m+1}$.

Let $u\in\mathcal{D}_{K_m}$ be a normalized eigenfunction to the eigenvalue $\lambda_{3^m+1}^m$, then $u$ is orthogonal to every $v$ that is constant on the $m$-cells (since this is a linear combination of eigenfunctions to lower eigenvalues). 
\begin{align*}
\lambda^m_{3^m+1}&=\mathcal{E}_{K_m}(u,u)\\
&=\mathcal{E}^\Sigma(u,u)+ \sum_{k=m+1}^\infty \frac 1{\gamma_k}\mathcal{D}^I_k(u,u)\\
&\geq \mathcal{K}r^{-m}\sum_{w\in\mathcal{A}^m} \mathcal{E}(u\circ G_w,u\circ G_{w})\\
&\stackrel{PI}{\geq}\mathcal{K}r^{-m} \sum_{w\in\mathcal{A}^m}C_{PI}\underbrace{\int_K |u\circ G_w - \overline{u\circ G_w}^{\mu^w}|^2d\mu^w}_{=:\star}
\end{align*}
 For $\star$ we have
\begin{align*}
&\int_K (u\circ G_w -\overline{u\circ G_w}^{\mu^w})^2d\mu^w\\
&\hspace*{1.5cm}=\int_K (u\circ G_w)^2d\mu^w - 2\int_K u\circ G_w\cdot \overline{u\circ G_w}^{\mu^w}d\mu^w + \underbrace{\int_K  (\overline{u\circ G_w}^{\mu^w})^2d\mu^w}_{\geq 0}\\
&\hspace*{1.5cm}\geq \frac 1{\mu(K_w)}\int_{K_w} (u)^2d\mu-2\frac 1{\mu(K_w)}\underbrace{\int_{K_w} u \cdot \overline{u\circ G_w}^{\mu^w}d\mu}_{=0, \text{ since u orth. on const.}}\\
&\hspace*{1.5cm}= \frac 1{\mu(K_w)} \int_{K_w} u^2d\mu 
\end{align*}
\begin{align*}
\Rightarrow \lambda^m_{3^m+1} &\geq \mathcal{K}r^{-m}\sum_{w\in\mathcal{A}^m} C_{PI}\frac 1{\mu(K_w)} \int_{K_w} u^2d\mu \\
&\geq r^{-m} \frac {\mathcal{K}\cdot C_{PI}}{\max \mu(K_w)}\int_K u^2d\mu \\
&\geq \frac{r^{-m}}{3^{-m}} \mathcal{K}\cdot C_{PI}=C_u\left(\frac 3r\right)^m
\end{align*}
We have, $\lambda_{3^m+1}^m\geq C_u (3/r)^m$, that means
 $$x< C_u(3/r)^m \Rightarrow N(\mathcal{E}_{K_m},\mathcal{D}_{K_m},x)\leq 3^m$$
For $x\geq C_u$ take $m\in \mathbb{N}$ such that $C_u(3/r)^{m-1}\leq x<  C_u(3r^{-1})^m$
\begin{align*}
\Rightarrow N(\mathcal{E}_{K_m},\mathcal{D}_{K_m},x)&\leq 3^m \leq 3\cdot 3^{m-1}= 3 \left( \left(\frac 3r\right)^{\frac{\ln(3)}{\ln(3/r)}}\right)^{m-1}\\
&=3 \left(\left( \frac 3r\right)^{m-1}\right)^{\frac{\ln(3)}{\ln(3/r)}}\leq 3\left( \frac x{C_u} \right)^{\frac{\ln(3)}{\ln(3/r)}}\\
&\leq \underbrace{3 C_u^{-{\frac{\ln(3)}{\ln(3/r)}}}}_{C_2^\prime:=} x^{\frac{\ln(3)}{\ln(3/r)}}
\end{align*}

\textbf{I.2: Line part $(\mathcal{E}_{J_m},\mathcal{D}_{J_m})$}\nopagebreak\\[0.2cm]
Due to the decoupling through the Neumann boundary conditions the domain and form split into
\begin{align*}
\mathcal{E}_{J_m}&=\bigoplus_{\begin{array}{c} w\in\mathcal{A}^n, n< m \\ i\in\{1,2,3\}\end{array} }\frac 1{\gamma_{|w|+1}} \int_0^1 \left(\frac{d(\cdot\circ \xi_{e^i_w})}{dx}\right)^2 d\mu\\
\mathcal{D}_{J_m}&=\bigoplus_{\begin{array}{c} w\in\mathcal{A}^n, n< m \\ i\in\{1,2,3\}\end{array} }H^1(e^i_w)
\end{align*}
Then it holds for the eigenvalue counting function that
\begin{align*}
N(\mathcal{E}_{J_m},\mathcal{D}_{J_m},x)=\sum_{\begin{array}{c} w\in\mathcal{A}^n, n< m \\ i\in\{1,2,3\}\end{array} }N\left(\frac 1{\gamma_{|w|+1}} \int_0^1 \left(\frac{d(\cdot\circ \xi_{e^i_w})}{dx}\right)^2 d\mu, H^1(e_w^i),x\right)
\end{align*}
The scaling parameter for the measure on the line part scales the integral in the following way:
\begin{align*}
\frac 1{\gamma_{|w|+1}}\int_0^1 \left(\frac{d(u\circ \xi_{e^i_w})}{dx}\right)d\mu =\frac 1{\gamma_{|w|+1}\frac 12  a\beta^{|w|}}\int_0^1\left(\frac{d(u\circ \xi_{e^i_w})}{dx}\right)^2dx
\end{align*}
Therefore there is a 1:1 correspondence of the eigenvalues between the standard Neumann Laplacian on $(0,1)$ and the restriction of the energy to one edge.
\begin{align*}
N\left(\frac 1{\gamma_{|w|+1}} \int_0^1 \left(\frac{d(\cdot\circ \xi_{e^i_w})}{dx}\right)^2 d\mu, H^1(e_w^i),x\right)=N(-\Delta_N|_{(0,1)},\tfrac 12 a\beta^{|w|}\gamma_{|w|+1}x)
\end{align*}
With 
\begin{align*}
N(-\Delta_N|_{(0,1)},x)\leq \frac 1\pi \sqrt{x}+1, \ \forall x\geq 0
\end{align*}
we get
\begin{align*}
N(\mathcal{E}_{J_m},\mathcal{D}_{J_m},x)&\leq \sum_{k=1}^m\sum_{j=1}^{3^k} N(-\Delta_N|_{(0,1)},\tfrac 12  a\beta^{k-1}\gamma_kx)\\
&\leq \sum_{k=1}^m\sum_{j=1}^{3^k} \frac 1\pi \sqrt{\tfrac 12 a\beta^{k-1}\gamma_k x}+1\\
&=\sum_{k=1}^m \frac {3^k}\pi\sqrt{\tfrac 12 a\beta^{k-1}\gamma_k x} + 3^k\\
&\leq\frac 32 (3^m-1)+ \frac{\sqrt{a}}{\sqrt{2}\pi}\sqrt{x}\sum_{k=1}^m \sqrt{9^k\beta^{k-1} \kappa_2 r^{k-1}}\\
&\leq \frac 32 3^m+  \frac{3\sqrt{a\kappa_2}}{\sqrt{2}\pi}\sqrt{x}\sum_{k=0}^{m-1} \sqrt{9\beta r}^{k}
\end{align*}
From here on we have to distinguish a few cases. For now assume that $\frac 1{9r}<\beta$:
\begin{align*}
N(\mathcal{E}_{J_m},\mathcal{D}_{J_m},x)&\leq \frac 32 3^m+ \frac{3\sqrt{a\kappa_2}}{\sqrt{2}\pi(\sqrt{9\beta r}-1)}\sqrt{9\beta r}^m \sqrt x
\end{align*}
For the fractal part we looked for the $m$ for which $C_u(3/r)^{m-1}\leq x <C_u(3/r)^m$. Therefore
\begin{align*}
N(\mathcal{E}_{J_m},\mathcal{D}_{J_m},x)&\leq \frac 32 3^m + \frac{3\sqrt{a\kappa_2}}{\sqrt{2}\pi(\sqrt{9\beta r}-1)} \sqrt{9\beta r}^m \sqrt{C_u \frac 3r}^m\\
&= \frac 32 3^m+ \frac{3\sqrt{a\kappa_2C_u}}{\sqrt{2}\pi(\sqrt{9\beta r}-1)}\sqrt{9\beta 3}^m
\end{align*}
Since $\beta <\frac 13$ we get a constant $C_2^{\prime \prime}$, such that for $x$ with $C_u(3/r)^{m-1}\leq x <C_u(3/r)^m$ we have
\begin{align*}
N(\mathcal{E}_{J_m},\mathcal{D}_{J_m},x)\leq C_2^{\prime \prime} \cdot 3^m
\end{align*}
For $\beta =\frac 1{9r}$ we can change to $\tilde \beta =\beta + \epsilon$ with $\frac 1{9r}< \tilde \beta <\frac 13$ and get the same results.\\

With the same calculations as for the fractal part we get the same order $\frac{\ln(3)}{\ln(3/r)}$ for the upper bound. That means for $x\geq C_u$ there exists a constant $C_2$, such that 
\begin{align*}
N_N(x)\leq C_2 x^{\frac{\ln(3)}{\ln(3/r)}}
\end{align*}

Now we go back to earlier and handle the case $\beta<\frac 19r$. In this case we have
\begin{align*}
N(\mathcal{E}_{J_m},\mathcal{D}_{J_m},x)&\leq \frac 32 3^m+  \frac{3\sqrt{a\kappa_2}}{\sqrt{2}\pi}\sqrt{x}\sum_{k=0}^{m-1} \sqrt{9\beta r}^{k}\\
&\leq \frac 32 3^m+  \frac{3\sqrt{a\kappa_2}}{\sqrt{2}\pi}\sqrt{x}\sum_{k=0}^{\infty} \sqrt{9\beta r}^{k}\\
&\leq \frac 32 3^m + \frac{3\sqrt{a\kappa_2}}{\sqrt{2}\pi}\frac 1{1-\sqrt{9\beta r}}\sqrt{x}
\end{align*}
The first part handles exactly as before to give the same order as in the fractal case and the latter part is of lower order, or the same order if $r=\frac 13$. Therefore we have the desired result.\\[.2cm]

\textbf{II: Lower estimate} \nopagebreak\\[0.2cm] 
The idea here is to successively add new Dirichlet boundary conditions on the points $V_m$ thus lowering the eigenvalue counting function. 
\begin{align*}
 \mathcal{D}_m^0&:=\{u\in \mathcal{D}^0 : u|_{V_m}\equiv 0\} \\
 \mathcal{D}^0_w&:=\{u\in\mathcal{D}_m^0: u|_{K_w^c}\equiv 0\},\ w\in\mathcal{A}^m\\
 \mathcal{D}^0_{e_i^w}&:=\{u\in\mathcal{D}_m^0: u|_{(e_i^w)^c} \equiv 0\},\ w\in\mathcal{A}^m, i\in \{1,2,3\}
\end{align*}

Then we have
\begin{lem} $(\mathcal{E}|_{\mathcal{D}^0_m\times \mathcal{D}^0_m},\mathcal{D}_m^0)$, $(\mathcal{E}|_{\mathcal{D}^0_w\times \mathcal{D}^0_w},\mathcal{D}_w^0)$ and $(\mathcal{E}|_{\mathcal{D}^0_{e_i^w}\times \mathcal{D}^0_{e_i^w}},\mathcal{D}^0_{e_i^w})$ are regular Dirichlet forms with discrete non negative spectrum.\label{lem62}\end{lem}
\begin{proof} The proof for $(\mathcal{E}|_{\mathcal{D}^0_m\times \mathcal{D}^0_m},\mathcal{D}_m^0)$ works just like the one of Lemma \ref{lem42} and the rest like Lemma \ref{lem61}.\end{proof}
We also get the following estimate 
$$ N(\mathcal{E}_m^0,\mathcal{D}_m^0, x)\leq N_D(x), \ \forall x\geq 0 $$

Due to the finite ramification and the condition, that the functions in $\mathcal{D}_m^0$ have to be zero in $V_m$, this domain splits into the domain restricted to the different parts.
\begin{align*}
\mathcal{D}_m^0&=\left(\bigoplus_{w \in\mathcal{A}^m} \mathcal{D}_w^0\right)\bigoplus \left(\bigoplus_{\begin{array}{c} w\in\mathcal{A}^n, n< m \\ i\in\{1,2,3\}\end{array} } \mathcal{D}^0_{e_i^w}\right)
\end{align*}
That means for the eigenvalue counting function $\forall x\geq 0$
\begin{align*}
\sum_{w \in \mathcal{A}^m}N(\mathcal{E}|_{\mathcal{D}_w^0\times \mathcal{D}_w^0},\mathcal{D}_w^0,x) + \sum_{\begin{array}{c} w\in\mathcal{A}^n, n< m \\ i\in\{1,2,3\}\end{array} } N(\mathcal{E}|_{\mathcal{D}_{e_i^w}^0\times \mathcal{D}_{e_i^w}^0}, \mathcal{D}_{e_i^w}^0,x)\leq N_D(x) 
\end{align*}

Again due to the decoupling, the individual eigenvalue counting functions can be calculated seperately.\\

\textbf{II.1: Fractal part $(\mathcal{E}|_{\mathcal{D}^0_w\times \mathcal{D}^0_w},\mathcal{D}^0_w)$}\nopagebreak\\[0.2cm]
We want to get an upper estimate on the first eigenvalue of $(\mathcal{E}|_{\mathcal{D}_w^0\times \mathcal{D}_w^0},\mathcal{D}_w^0)$ which is positive due to the Dirichlet boundary conditions. This estimate gives us a lower estimate for $N(\mathcal{E}|_{\mathcal{D}_w^0\times \mathcal{D}_w^0},\mathcal{D}_w^0,x)$. The first eigenvalue can be calculated via the following fact
\begin{align*}
\lambda_1^w &= \inf_{u\in\mathcal{D}^0_w} \frac{\mathcal{E}(u,u)}{||u||^2}\\
\Rightarrow \lambda_1^w &\leq \frac{\mathcal{E}(u,u)}{||u||^2}, \text{ for each }u\in \mathcal{D}^0_w
\end{align*}
The idea is to find an $u\in\mathcal{D}^0_w$ which is "good enough".\\
\begin{figure}[H]
\centering
\includegraphics[scale=0.1]{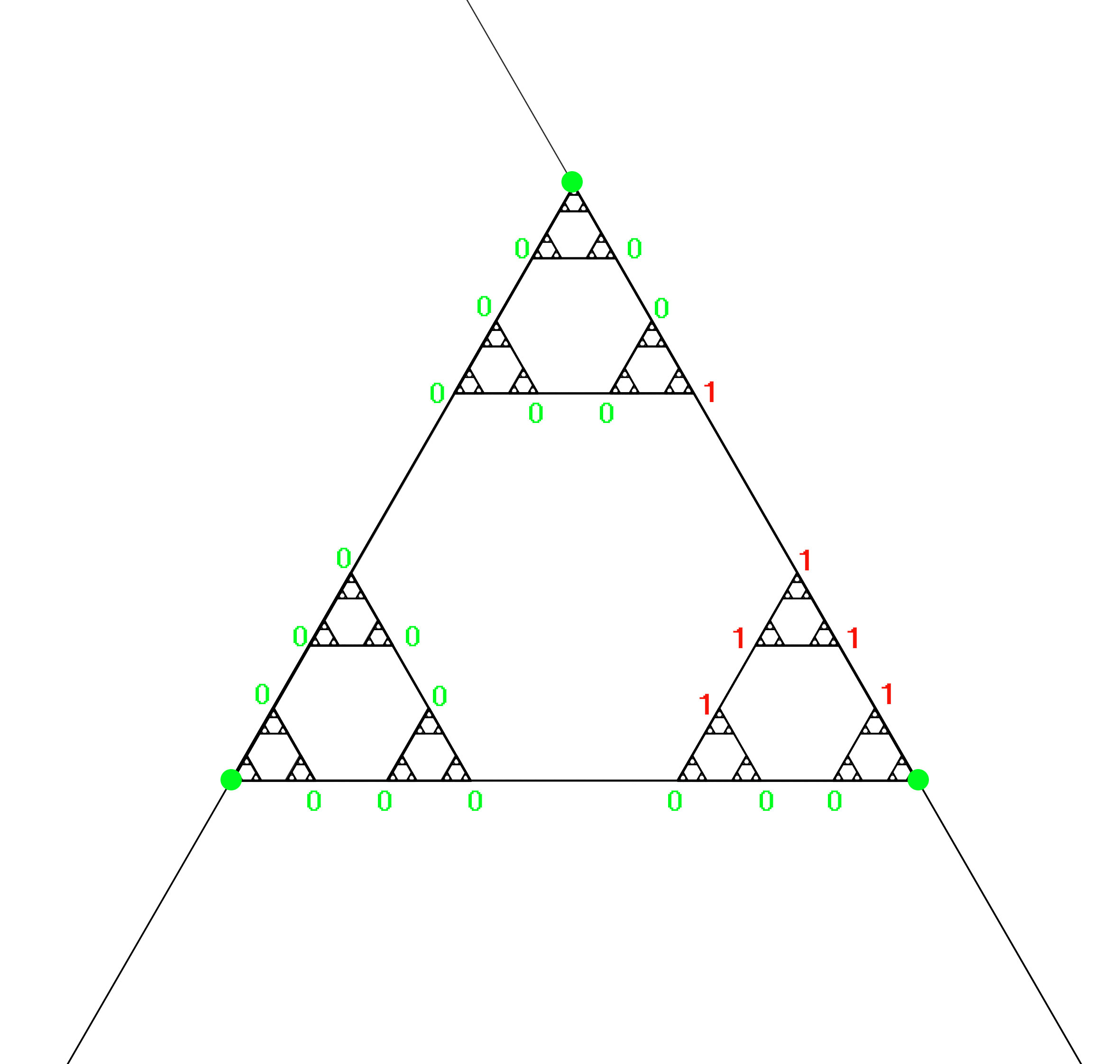}
\caption{Construction of $u_m$}
\label{dirich}
\end{figure}
In $K_w$ we look for the biggest cell where there are no Dirichlet boundary conditions. This is a $m+2$-cell (choose any of those). Set
\begin{align*}
\tilde u_m|_{V_{m+2}}=\begin{cases} 1&, \ \text{on this cell and any adjoined vertices}\\0&, \ \text{anywhere else}\end{cases}
\end{align*}
Then extend $\tilde u_m$ harmonic to $u_m\in \mathcal{D}_w^0$. The energy of this function is calculated by
\begin{align*}
\mathcal{E}(u_m,u_m)&=6 \cdot \delta_{m+2}^{-1}\\
&\leq \frac 6{\kappa_1} r^{-(m+2)}
\end{align*}

We need a lower estimate for the $L^2$-norm of $u_m$ to get an upper estimate of $\lambda_1^w$. There is a $m+2$-cell $K_{\tilde w}$ in $K_w$ with $|\tilde w|=m+2$ where $u_m$ is constant $1$. Therefore
\begin{align*}
||u_m||^2&=\int_{K_w} |u_m|^2 d\mu \\
&\geq \int_{K_{\tilde w}}\underbrace{|u_m|^2}_{=1}d\mu\\
&=\mu(K_{\tilde w})
\end{align*}
$$\Rightarrow \lambda_1^w \leq \frac 6{\kappa_1} \frac{r^{-(m+2)}}{\mu(K_{\tilde w})}$$
For the mass of $m$-cells we have
\begin{align*}
3^m\mu(K_w)&=3^m \frac 12 (\mu_f+\mu_l)(K_w)\\
&\geq \frac 12 3^m \mu_f(K_w)\\
&=\frac 12
\end{align*}
Therefore
\begin{align*} \lambda_1^w &\leq \frac {2\cdot 6}{\kappa_1} (3r^{-1})^{m+2} \\
&= \underbrace{\frac{12(3r^{-1})^2}{\kappa_1}}_{C_l:=}(3r^{-1})^m
\end{align*}
For $x\geq C_l(3r^{-1})$ choose $m\in \mathbb{N}$ such that
$$ C_l(3r^{-1})^m \leq x < C_l(3r^{-1})^{m+1}$$
For these $x$ it holds, that there is at least one eigenvalue smaller than $x$ from $(\mathcal{E}|_{\mathcal{D}_w^0\times \mathcal{D}_w^0},\mathcal{D}_w^0)$:
\begin{align*}
N(\mathcal{E}|_{\mathcal{D}_w^0\times \mathcal{D}_w^0},\mathcal{D}_w^0,x)&\geq 1\\
\Rightarrow \sum_{w \in \mathcal{A}^m}N(\mathcal{E}|_{\mathcal{D}_w^0\times \mathcal{D}_w^0},\mathcal{D}_w^0,x)&\geq 3^m =\frac 13 \left( (r^{-1}3)^{m+1}\right)^{\frac{\ln(3)}{\ln(r^{-1}3)}}\\
&\geq \underbrace{\frac 13 C_l^{\frac{\ln(3)}{\ln(r^{-1}3)}}}_{C_1:=}\cdot x^{\frac{\ln(3)}{\ln(r^{-1}3)}}
\end{align*}

\textbf{II.2 Line part $(\mathcal{E}|_{\mathcal{D}_{e_i^w}^0\times \mathcal{D}_{e_i^w}^0}, \mathcal{D}_{e_i^w}^0)$}\nopagebreak\\[.1cm]
In the previous calculations we saw that the fractal part already gives a lower bound with the same order as the upper bound. Therefore the influence of the line part can not be bigger than the fractal part. We can use the trivial estimate  
\begin{align*}
\sum_{ w\in\mathcal{A}^n, n\leq m-1, i} N(\mathcal{E}|_{\mathcal{D}_{e_i^w}^0\times \mathcal{D}_{e_i^w}^0}, \mathcal{D}_{e_i^w}^0,x)\geq 0
\end{align*}
This suffices to show the desired result.

\section{Generalization}\label{chap7}
If we no longer demand that $r_m\rightarrow r$, we need other conditions.

Let \begin{align*}r^\ast&:=\limsup_{m\rightarrow\infty} r_m<\frac 35 \\r_\ast&:=\liminf_{m\rightarrow\infty} r_m\geq \frac 13 \end{align*}
The conditions on the sequence of matching pairs should be, that the elements of the sequence which are above $r^\ast$ and below $r_\ast$ behave nicely. This could be expressed as
\begin{align*}
\sum_{m: r_m>r^\ast} 1-(r^\ast)^{-1}r_m < \infty\\
\sum_{m: r_m<r_\ast} 1-(r_\ast)^{-1}r_m <\infty
\end{align*}
With this we should get upper and lower estimates on $\delta_m$ and thus for the energy.\\[.2cm]
The conditions are equivalent to
\begin{align*}
0<\prod_{m: r_m>r^\ast}(r^\ast )^{-1}r_m <\infty\\
0<\prod_{m: r_m<r_\ast}(r^\ast )^{-1}r_m <\infty
\end{align*}
Then we get constants $\kappa^\ast,\kappa_\ast$ with
\begin{align*}
1 &\leq \prod_{k\leq m: r_k>r^\ast}(r^\ast )^{-1}r_k\leq \kappa^\ast, \ \forall m\\
\kappa_\ast &\leq \prod_{k\leq m: r_k<r_\ast}(r_\ast )^{-1}r_k\leq 1,\quad \forall m
\end{align*}
Since $r_k>r^\ast$ in the first product we have that $k^\ast >1$ and analogously $k_\ast<1$. With that we get estimates for $\delta_m$.
\begin{align*}
\delta_m&=r_1\cdots r_m\\
&=\left(\prod_{k\leq m: r_k\leq r^\ast} r_k \right)\cdot \left(\prod_{k\leq m : r_k>r^\ast} r_k\right)\\
&\leq (r^\ast)^{\#\{k\leq m: r_k\leq r^\ast\}}\cdot \kappa^\ast (r^\ast)^{\#\{k\leq m : r_k>r^\ast\}}\\
&\leq \kappa_2 (r^\ast)^m
\end{align*}
There is also a bound from below:
\begin{align*}
\delta_m&=r_1\cdots r_m\\
&=\left(\prod_{k\leq m: r_k\geq r_\ast} r_k \right)\cdot \left(\prod_{k\leq m : r_k<r_\ast} r_k\right)\\
&\geq (r_\ast)^{\#\{k\leq m: r_k\geq r_\ast\}}\cdot \kappa_{\ast} (r_\ast)^{\#\{k\leq m : r_k<r_\ast\}}\\
&\geq \kappa_{\ast} (r_\ast)^m
\end{align*}
A quick calculation shows, that these equalities also hold for every product of $m$ different $r_i$.
This leads for $\tilde{\gamma}_k=r_1\cdots r_{m+k-1}\rho_{m+k}$ to
\begin{align*}
\tilde\gamma_k&\leq \mathcal{K}_1(r^\ast)^m \gamma_k
\end{align*}
and
\begin{align*}
(r^\ast)^{-m}\sum_{w\in\mathcal{A}^m} \mathcal{E}_\mathcal{R}^\Sigma(u\circ G_w,u\circ G_w)\leq &\mathcal{E}_\mathcal{R}^\Sigma(u,u)\leq (r_\ast)^{-m}\sum_{w\in\mathcal{A}^m} \mathcal{E}_\mathcal{R}^\Sigma(u\circ G_w,u\circ G_w)\\
\frac 1{\mathcal{K}_1} (r^\ast)^{-m}\sum_{w\in\mathcal{A}^m}\mathcal{E}_\mathcal{R}^I(u\circ G_w,u\circ G_w)\leq &\mathcal{E}_\mathcal{R}^I(u,u)
\end{align*}
for the whole energy with $\mathcal{K}:=\min\{1,1/\mathcal{K}_1\}$
\begin{align*}
\mathcal{E}_\mathcal{R}(u,u)\geq \mathcal{K} (r^\ast)^m \sum_{w\in\mathcal{A}^m} \mathcal{E}_\mathcal{R}(u\circ G_w,u\circ G_w)
\end{align*}
To get these estimates for $r^\ast=\frac 35$ we again need the monotonic decrease of $(\rho_k)_{k\geq 1}$ to get the estimates for all $\tilde \gamma_k$. But then we are in the case where $r_m\rightarrow \frac 35$. \\

These estimates are enough to apply the same proofs as before to get more general results. For the Hausdorff-Dimension in resistance metric we get the following result:\\
\begin{theo} Let $\mathcal{R}=(r_i,\rho_i)_{i\geq 1}$ be a sequence of matching pairs that fulfills the conditions (of chapter \ref{chap7}), then
\begin{align*}
\frac{\ln(3)}{-\ln(r_\ast)}\leq \dim_{H,R_\mathcal{R}}(K)\leq \frac{\ln(3)}{-\ln(r^\ast)}
\end{align*}\label{theo71}
\end{theo}
\begin{proof}  The proofs of Lemma \ref{lem32} and Lemma \ref{lem33} work exactly the same with $r^\ast$ resp. $r_\ast$ instead of $r$. These Lemmata are exactly responsible for the upper and lower bound in the proof of \cite[Theo. 2.4]{kig95}. \end{proof}
The results of chapter \ref{chap5} can also be generalized to these weaker conditions on the sequences of matching pairs. 
\begin{theo} Let $\mathcal{R}=(r_i,\rho_i)_{i\geq 1}$ be a sequence of matching pairs that fulfills the conditions (of Chapter \ref{chap7}), then there exist constants $0<C_1,C_2<\infty$ and $x_0>0$, such that for all $x\geq x_0$:
\begin{align*}
C_1x^{\frac 12 d_{S,1}^{\mathcal{R}}(K)}\leq N_D^{\mu,\mathcal{R}}(x)\leq N_N^{\mu,\mathcal{R}}(x)\leq C_2 x^{\frac 12 d_{S,2}^{\mathcal{R}}(K)}
\end{align*}
with
\begin{align*}
d_{S,1}^\mathcal{R}(K)=\frac{\ln 9}{\ln 3- \ln r_\ast}, \qquad d_{S,2}^\mathcal{R}(K)=\frac{\ln 9}{\ln 3- \ln r^\ast}
\end{align*}
\label{theo72}\end{theo}
\begin{proof} The proof in chapter \ref{chap6} works again if we use the estimates of $\delta_m, \gamma_m$ and $\mathcal{E}_\mathcal{R}$ from above and change $r$ to $r^\ast$ resp. $r_\ast$ for the upper resp. lower bound. \end{proof}
\section*{Acknowledgements}
I would like to thank Prof. Jun Kigami and Dr. Patricia Alonso-Ruiz for fruitful discussions during the Fractals 6 Conference at Cornell. In particular chapter \ref{chap7} developed from these conversations. 

\end{document}